\documentclass[11pt]{article}
\usepackage{amsmath,amssymb}

 \def\qed{\unskip\quad \hbox{\vrule\vbox to 6pt {\hrule width
          4pt\vfill\hrule}\vrule} }
\newcommand{\bez}{\nopagebreak\hspace*{\fill}
          \nolinebreak$\qed$\vspace{2ex}\par}
\newenvironment{proof}{\vspace{1ex}
    \vspace*{10mm}\vspace{-10mm}\par\noindent{\bf Proof.}\nopagebreak
     \par}
     {\nopagebreak\linebreak[0]\hspace*{\fill} $\qed$\pagebreak[0]\vspace{2ex}\par}

\newtheorem{Th}{Theorem}
\newtheorem{Prop}{Proposition}
\newtheorem{Lemma}{Lemma}
\newtheorem{Remark}{Remark}

\newtheorem{Cor}{Corollary}

\newcommand{\tor}{{\Bbb{T}}}

\newcommand{\cs}{{\cal S}}

\newcommand{\ck}{{\cal K}}
\newcommand{\cb}{{\cal B}}
\newcommand{\ca}{{\cal A}}
\newcommand{\cc}{{\cal C}}

\newcommand{\R}{{\Bbb{R}}}
\newcommand{\T}{{\Bbb{T}}}
\newcommand{\C}{{\Bbb{C}}}
\newcommand{\Z}{{\Bbb{Z}}}
\newcommand{\N}{{\Bbb{N}}}
\newcommand{\xbm}{(X,{\cal B},\mu)}
\newcommand{\ycn}{(Y,{\cal C},\nu)}
\newcommand{\zde}{(Z,{\cal D},\eta)}
\newcommand{\tfs}{T_{\varphi,{\cal S}}}
\newcommand{\stfs}{\sigma_{T_{\varphi,{\cal S}}}}
\newcommand{\lf}{\Lambda_{\varphi}}
\newcommand{\svf}{\Sigma_{\varphi}}
\newcommand{\vhf}{V_{\chi\circ\varphi}}
\newcommand{\ov}{\overline}
\newcommand{\til}{\widetilde}
\newcommand{\beq}{\begin{equation}}
\newcommand{\eeq}{\end{equation}}
\newcommand{\vep}{\varepsilon}
\newcommand{\va}{\varphi}
\newcommand{\si}{\sigma}
\newcommand{\ot}{\otimes}
\newcommand{\la}{\lambda}

\begin{document}

\title{Lifting mixing properties by Rokhlin cocycles}
\author{M. Lema\'nczyk\thanks{Research
partly supported by  Polish MNiSzW grant N N201 384834} \and F.
Parreau}
\date{}

\maketitle


\begin{abstract} We study the problem of lifting
various mixing properties from a base automorphism $T\in {\rm
Aut}\xbm$ to skew products of the form $\tfs$, where $\va:X\to G$
is a cocycle with values in a locally compact Abelian group $G$,
$\cs=(S_g)_{g\in G}$ is a measurable representation of $G$ in
${\rm Aut}\ycn$ and $\tfs$ acts on the product space
 $(X\times Y,\cb\ot\cc,\mu\ot\nu)$ by
$$\tfs(x,y)=(Tx,S_{\va(x)}(y)).$$
It is also shown that whenever $T$ is ergodic (mildly mixing,
mixing) but $\tfs$ is not ergodic (is not mildly mixing, not
mixing), then on a non-trivial factor $\ca\subset\cc$ of $\cs$ the
corresponding Rokhlin cocycle $x\mapsto S_{\va(x)}|_{\ca}$ is a
coboundary (a quasi-coboundary).
\end{abstract}

\section*{Introduction} Given an ergodic automorphism $T$ of a
standard Borel space $\xbm$ we can study various extensions
$\widetilde{T}$ of it. Among such extensions a special role is
played by so called compact group extensions or, more generally,
isometric extensions (see \cite{Fu3}, \cite{Gl} and~\cite{Zi}). In
particular, one can ask which ergodic properties of $T$ are lifted
by isometric extensions. The two papers\footnote{In \cite{Ru1} it
is proved that Bernoullicity is lifted whenever the extension is
weakly mixing, while in \cite{Ru2} it is shown that mixing
(multiple mixing) lifts whenever the extension is weakly mixing.}
by Dan Rudolph \cite{Ru1} and \cite{Ru2} are beautiful examples of
the mechanism that once the extension enjoys some ``minimal''
ergodic property then it shares some strong ergodic properties
assumed to hold for its base. By iterating the procedure of taking
isometric extensions we can hence lift ergodic properties of $T$
to weakly mixing  distal extensions of it.

The notion complementary to distality is relative weak mixing
\cite{Fu3}, \cite{Gl},  \cite{Zi} and a natural question arises
what happens with lifting ergodic properties from $T$ to
$\widetilde{T}$ when $\widetilde{T}$ is relatively weakly mixing
over the factor $T$. This, by Abramov-Rokhlin's theorem
\cite{Ab-Ro}, leads to the study of so called Rokhlin cocycle
extensions which are automorphisms of the form
$\widetilde{T}=T_\Theta$ acting on $(X\times
Y,\cb\ot\cc,\mu\ot\nu)$ by the formula
$$
T_\Theta(x,y)=(Tx,\Theta_x(y)),$$ where $\Theta:X\to{\rm Aut}\ycn$
is measurable\footnote{The map $\Theta$ is often called a Rokhlin
cocycle.}. Since the above formula describes all possible
(ergodic) extensions of $T$, it is hard to expect interesting
theorems on such a level of generality -- one has to specify
subclasses of Rokhlin cocycles for which one can obtain some
results. We will focus on the following class.

Let $G$ be a second countable locally compact Abelian (LCA) group.
Assume that we have a measurable action $\cs$ of this group given
by $g\mapsto S_g\in \mbox{Aut}\ycn$. Let $\va:X\to G$ be a
cocycle. The automorphism $\tfs$ acting on $(X\times Y,{\cal
B}\otimes{\cal C},\mu\otimes\nu)$ given by
$$
\tfs(x,y)=(Tx,S_{\va(x)}(y))
$$
will be called the {\em Rokhlin $(\va,\cs)$-extension}\footnote{We
would like to emphasize that, as noticed in \cite{Da-Le}, if we
admit $G$ to be non-Abelian locally compact, then each ergodic
extension $\widetilde{T}=T_{\Theta}$ is of the form $\tfs$; more
specifically, a general Rokhlin cocycle $x\mapsto\Theta(x)$ is
cohomologous to a cocycle $x\mapsto S_{\va(x)}$ for some $G,\va$
and $\cs$.} of $T$.

A systematic study of the problem of lifting ergodic properties
from $T$ to $\tfs$ was originated by D. Rudolph in \cite{Ru3}.
Since then, extensions $\tfs\to T$  have been studied in numerous
papers, see  {\em e.g.}  \cite{Da-Le}, \cite{Gl0}, \cite{Gl},
\cite{Gl-We}, \cite{Le-Le}, \cite{Le-Pa}, \cite{Ro} and~\cite{Ry}.

The present paper is a  continuation of investigations from
\cite{Le-Le} and \cite{Le-Pa}, and, due to a new approach
presented here, makes them complete. This new approach is based on
a harmonic analysis result from \cite{Ho-Me-Pa}, and it consists
in showing that given an action $\cs=(S_g)_{g\in G}$ of a second
countable LCA group $G$ on a probability standard Borel space
$\ycn$ and a saturated Borel subgroup $\Lambda\subset\widehat{G}$,
the spectral space of functions in $L^2\ycn$ whose spectral
measures are concentrated on $\Lambda$ is the $L^2$-space of an
$\cs$-invariant sub-$\sigma$-algebra ${\cal A}\subset{\cal C}$  (a
measure-theoretic factor of $\cs$). This will systematically be
used in our study because the group of $L^\infty$-eigenvalues of
the Mackey $G$-action associated to $T$ and $\va$ is saturated and
hence yields an $\cs$-factor.

Using that we will prove natural necessary and sufficient
conditions for weak mixing of $\tfs$ and relative weak mixing of
$\tfs$ over $T$. We also compute possible eigenvalues of $\tfs$
and determine the relative Kronecker factor whenever $\tfs$ is
ergodic. The idea of a factor determined by a saturated group
allows us to prove that if $T$ is ergodic but $\tfs$ is not, then
the Rokhlin cocycle $x\mapsto S_{\va(x)}|_{\ca}$ is a coboundary
as a cocycle taking values in ${\rm Aut}(\ca)$, where $\ca$ is the
non-trivial factor of $\cs$ corresponding to the above-mentioned
eigenvalue group. Finally, by replacing coboundary by
quasi-coboundary, a similar conclusion is achieved when $T$
is mildly mixing but $\tfs$ is not, and when $T$ is mixing but
$\tfs$ is not.

Another tool explored here is a use of mixing sequences of
weighted unitary operators, that is, of operators on $L^2\xbm$
given by the formula
$$
f\mapsto \xi\cdot f\circ T\;\;\mbox{for each $f\in L^2\xbm$}
$$
determined by a measurable $\xi:X\to\T$ and an automorphism $T$.
This, in particular, will solve the problem of lifting mild mixing
property, and complete the picture from \cite{Le-Pa} of lifting
mixing and multiple mixing.

\section{Preliminaries}
We briefly recall basic definitions, some known results and fix
notation for the rest of the paper.

\subsection{Self-joinings of an  automorphism, relative concepts}
\label{jeden} Assume that $T$ is an  automorphism of a standard
probability Borel space $\xbm$, which we denote $T\in{\rm
Aut}\xbm$ \footnote{We shall also denote  by $T$ the unitary
operator $f\mapsto f\circ T$ on $L^2\xbm$.}. Denote by $J(T)$  the
set of self-joinings of $T$, that means the set of $T\times
T$-invariant probability measures on $(X\times X,{\cal
B}\otimes{\cal B})$ whose both marginals are equal to $\mu$. To
each self-joining $\eta\in J(T)$ one associates a Markov operator
\footnote{A linear bounded operator $\Phi$ of $L^2\xbm$ is called
{\em Markov} if $\Phi(1)=1=\Phi^\ast(1)$ and $\Phi f\geq0$
whenever $f\geq0$. Notice also that we always have $\|\Phi_\eta
f\|\leq \|f\|$ and thus $\|\Phi_\eta\|=1$.} $\Phi_\eta$ of
$L^2\xbm$ given by
$$
\int_X \Phi_\eta f(y)g(y)\,d\mu(y)=\int_{X\times
X}f(x)g(y)\,d\eta(x,y)$$ for each $f,g\in L^2\xbm$. Moreover, the
$T\times T$--invariance of $\eta$ means that \beq\label{mar2}
\Phi_\eta \circ T=T\circ \Phi_\eta. \eeq On the other hand each
Markov operator $\Phi$ on $L^2\xbm$ for which (\ref{mar2}) holds
determines a self-joining $\eta_\Phi$ by the formula
$$\eta_\Phi(A\times B)=\int_B\Phi(1_A)\,d\mu$$ for each $A,B\in\cb$.
Then \beq\label{odp} \Phi=\Phi_{\eta_\Phi}\;\mbox{ and
}\;\eta=\eta_{\Phi_\eta}.\eeq Therefore the set $J(T)$ can
naturally be identified with the set ${\cal J}(T)$ of Markov
operators on $L^2\xbm$ satisfying (\ref{mar2}). The set ${\cal
J}(T)$ is a closed subset in the weak operator topology and hence
it is compact. Thus
$$
\Phi_{n}\to\Phi \;\;\mbox{iff}\;\;
\langle\Phi_{n}f,g\rangle\to\langle\Phi f,g\rangle\;\;\mbox{for
each } f,g\in L^2\xbm.$$ By transferring the weak operator
topology via~(\ref{odp}) we obtain the weak topology on $J(T)$ and
$$
\eta_n\to\eta \;\;\mbox{iff}\;\; \eta_n(A\times B)\to\eta(A\times
B)\;\;\mbox{for each } A, B\in{\cal B}.$$ Since the composition of
two Markov operators is Markov, ${\cal J}(T)$ is a compact
semitopological semigroup. By the same token, $J(T)$ is also a
compact semitopological semigroup
($\eta_1\circ\eta_2:=\eta_{\Phi_{\eta_1}\circ\Phi_{\eta_2}}$).

Given a {\em factor} \footnote{Up to a little abuse of notation,
we define the {\em factor system} $T|_{\cal
A}:\;(X/\ca,\ca,\mu|_{\ca})\to(X/\ca,\ca,\mu|_{\ca})$ in which
cosets $\ov{x}\in X/\ca$ are given by those points which cannot be
distinguished by the sets from $\ca$; then $T|_{\cal
A}(\ov{x})=\ov{Tx}$.}
, {\em i.e.\ }a $T$-invariant sub-$\sigma$-algebra ${\cal
A}\subset{\cal B}$, let
$$
\mu=\int_{X/{\cal A}}\delta_{\ov{x}}\otimes
\mu_{\ov{x}}\,d\mu(\ov{x})$$ be the disintegration of $\mu$ over
the factor $\cal A$. By setting
$$
\mu\otimes_{\cal A}\mu=\int_{X/{\cal A}}\delta_{\ov{x}}\otimes
\mu_{\ov{x}}\ot\mu_{\ov{x}}\;d\mu(\ov{x}).
$$
we obtain a self-joining $\mu\otimes_{\cal A}\mu$ which is often
called the {\em relative product over} $\ca$. Note that
$\mu\ot_{\ca}\mu|_{\ca\ot\ca}=\Delta_{\cal A}$, where
$\Delta_{\ca}(A_1\times A_2)=\mu(A_1\cap A_2)$ for each
$A_1,A_2\in\ca$. Moreover, we have $\Phi_{\mu\otimes_{\cal A}\mu}=
E(\cdot|{\cal A})$.

Assume additionally that $T$ is ergodic. Then we can speak about
ergodic self-joinings of $T$ and the set of such joinings will be
denoted by $J^e(T)$. By ${\cal J}^e(T)$ we denote the subset of
${\cal J}(T)$ corresponding to $J^e(T)$. The elements of ${\cal
J}^e(T)$ are exactly the extremal points in the natural simplex
structure of ${\cal J}(T)$. Recall that $T$ is said to be {\em
relatively weakly mixing} over a factor $\cal A$ if $E(\cdot|{\cal
A})\in{\cal J}^e(T)$.

The notion which is complementary to relative weak mixing is the
concept of relative Kronecker factor \cite{Fu3}, \cite{Zi}. More
precisely, if $\ca$ is a factor then the {\em relative Kronecker
factor} $\ck(\ca)$ (of $T$ over $T|_{\ca}$) is the smallest
$\sigma$-algebra making all relative eigenfunctions\footnote{By a
{\em relative} (with respect to $\ca$) {\em eigenvalue} of $T$ one
means an $\ca$-measurable map $c:(X/\ca,\ca,\mu|_{\ca})\to U(n)$
for which there is  $M:\xbm\to\C^n$ satisfying the following:
\beq\label{def1} c(\ov{x}) \left(\begin{array}{c}
M_1(x)\\
\cdots\\
M_n(x)\end{array}\right)= \left(\begin{array}{c}
M_1(Tx)\\
\cdots\\
M_n(Tx)\end{array}\right) \;\;\mbox{for a.e.}\;x\in X, \eeq
\beq\label{def2} M_i\perp_{\ca}M_j\;\;\mbox{for}\;i\neq
j\;\;\mbox{and}\; E(|M_{i}|^2|\ca)=1,\; i,j=1,\ldots,n. \eeq The
map $M$ satisfying~(\ref{def1}) and~(\ref{def2}) is called a {\em
relative eigenfunction corresponding to} $c$.} measurable
($\ca\subset\ck(\ca)$).

For more about joinings or relative concepts in ergodic theory,
see e.g.\ \cite{Fu3}, \cite{Gl}, \cite{Ka-Th}, \cite{Ry} and
\cite{Zi}.

\subsection{$G$-actions}
Assume that $G$ is a second countable LCA group. By a $G$-action
$\cs=(S_g)_{g\in G}$ we mean a {\em measurable representation} of
$G$ on a probability standard Borel space $\ycn$, that is a group
homomorphism $g\mapsto S_g$, $G\to {\rm Aut}\ycn$. Then we also
denote by  $\cs=(S_g)_{g\in G}$ the associated unitary
representation of $G$ on $L^2\ycn$, which is continuous.
 For each $f\in L^2\ycn$, by $\sigma_{f,\cs}$ (or
$\sigma_f$ is $\cs$ is understood) we denote the {\em spectral
measure} of $f$, {\em i.e.} the measure on the character group
$\widehat{G}$ \footnote{Since $G$ is second countable LCA, also
$\widehat{G}$ is second countable LCA.} determined by the Fourier
transform\footnote{By Pontryagin Duality Theorem, the character
group of $\widehat{G}$ has a natural identification with $G$.}
$$
\widehat{\sigma}_{f,\cs}(g):
=\int_{\widehat{G}}\chi(g)\,d\sigma_{f,\cs}(\chi)= \int_Yf\circ
S_g\cdot\ov{f}\,d\nu.
$$
We denote $G(f)=\ov{\rm span}\{S_gf:g\in G\}$. Then the
correspondence $f\to 1_{\widehat{G}}$ yields the canonical
isomorphism of $\cs|_{G(f)}$ with the representation ${\cal
V}_{\sigma_f}=(V^{\sigma_f}_g)_{g\in G}$ of
$L^2(\widehat{G},\cb(\widehat{G}),\sigma_f)$, where
$V^{\sigma_f}_gj(\chi)=\chi(g)j(\chi)$. The maximal spectral type
of $\cs$ on $L^2_0\ycn$ (the subspace of zero mean function in
$L^2\ycn$) will be denoted by $\sigma_{\cs}$ \footnote{Formally
speaking, it is the class of equivalence of measures which are
maximal spectral measures but in what follows we abuse the
vocabulary and often speak about a given measure as the maximal
spectral type. }.

For more about the spectral theory of $G$-actions, see {\em e.g.}
\cite{Ka-Th}, \cite{Le}.

Suppose that $\cs_i=(S_g^{(i)})_{g\in G}$ is a $G$-action on
$(Y_i,{\cal C}_i,\nu_i)$, $i=1,2$. By a {\em joining} of these two
$G$-actions we mean an $(S^{(1)}_g\times S_g^{(2)})_{g\in
G}$-invariant measure on $(Y_1\times Y_2,{\cal C}_1\otimes {\cal
C}_2)$ with projections $\nu_1$ and $\nu_2$
respectively\footnote{Slightly generalizing Section~\ref{jeden},
$\eta$  determines a Markov intertwining operator
$\Phi_\eta:L^2(Y_1,\cc_1,\nu_1)\to L^2(Y_2,\cc_2,\nu_2)$; the
correspondence similar to~(\ref{odp}) also takes place.}. Recall
that $\cs_1$ and $\cs_2$ are called {\em disjoint} (in the sense
of Furstenberg \cite{Fu}) if the only possible joining between
them is product measure. We then write $\cs_1\perp\cs_2$. It is
well-known \cite{Ha-Pa} that
$$
\sigma_{\cs_1}\perp\sigma_{\cs_2}\Rightarrow \cs_1\perp\cs_2.
$$

Denote by $M(\widehat{G})$ the convolution Banach algebra of all
complex Borel measures on $\widehat{G}$ \footnote{Since
$\widehat{G}$ is Polish, all members of $M(\widehat{G})$ are
regular measures.}. Let $M^+(\widehat{G})\subset M(\widehat{G})$
(respectively $M^{+,1}(\widehat{G})\subset M(\widehat{G})$)
consists of nonnegative members of $M(\widehat{G})$ (of all
probability measures in $M(\widehat{G})$).

Assume that $G$ is not compact. Recall that $\sigma\in
M^+(\widehat{G})$ is called {\em Dirichlet} if
$\limsup_{g\to\infty}|\widehat{\sigma}(g)|=\sigma(\widehat{G})$
(or equivalently if there exists a sequence $g_n\to\infty$ in $G$
such that $\widehat{\sigma}(g_n)\to \sigma(\widehat{G})$).

\subsection{$G$-valued cocycles for an ergodic automorphism}
Assume that $T\in{\rm Aut}\xbm$. Let $G$ be a second countable LCA
group\footnote{Here and all over the paper we use multiplicative
notation.}. Let $\va:X\to G$ be measurable. It determines a
cocycle $\va(n,x)=\va^{(n)}(x)$\footnote{$\va(\cdot,\cdot)$
satisfies the cocycle identity
$\va(m+n,\cdot)=\va(m,\cdot)\cdot\va(n,T^m\cdot)$; it is often
$\va$ itself which is called a cocycle. A cocycle $\va:X\to G$ is
called a {\em coboundary} if $\va=f/f\circ T$ for a measurable
$f:X\to G$. If two cocycles differ by a coboundary then they are
called {\em cohomologous}. A cocycle is said to be a {\em
quasi-coboundary} if it is cohomologous to a constant cocycle.} by
the following formula
$$ \va^{(n)}(x)=\left\{\begin{array}{ll}
\va(x)\cdot\va(Tx)\cdot\ldots\cdot\va(T^{n-1}x) &\mbox{if } n>0\\
1 &\mbox{if }  n=0\\
(\va(T^{n}x)\cdot\ldots\cdot \va(T^{-1}x))^{-1} &\mbox{if } n<0.
\end{array}\right.
$$
Let us recall now the definitions of two groups related to $T$ and
$\va$ that play basic role in the study of Rokhlin cocycle
extensions (studied in Section~\ref{TRZY}).

{\bf The group} $\lf$: This is a Borel subgroup of $\widehat{G}$
defined as
$$
\lf=\{\chi\in\widehat{G}:\:\chi\circ\va=\xi/\xi\circ T\;\mbox{for
a measurable}\;\xi:X\to\tor\}\footnote{We denote
$\tor=\{z\in\C:\:|z|=1\}$.}.
$$
This group turns out to be the group of $L^\infty$-eigenvalues of
the Mackey action (of $G$) associated to the cocycle $\va$ (see
{\em e.g.} \cite{Aa}, \cite{Ha-Os}, \cite{Ka}, \cite{Le-Pa}).

{\bf The group} $\svf$: This is a Borel subgroup of $\widehat{G}$
defined as
$$
\svf=\{\chi\in\widehat{G}:\:\chi\circ\va=c\cdot\xi/\xi\circ
T\;\mbox{for a measurable}\;\xi:X\to\tor\;\mbox{and}\;c\in\tor\}.
$$

\section{Tools}\label{DWA}
In this section we will present tools that will be needed to prove
lifting of various properties by Rokhlin cocycles (see
Section~\ref{TRZY}). Some of the results that will be presented
here  are new and seem to be of independent interest (see
Section~\ref{dwajeden}).

\subsection{Idempotents in $J(T)$}
Assume that $T\in{\rm Aut}\xbm$. Then the  closure of the group
$\{T^j:\:j\in \Z\}$ in ${\cal J}(T)$, denoted by $\ov{\{T^j:\:j\in
\Z\}}$, is a closed subsemigroup of ${\cal J}(T)$ and therefore
\beq\label{mar4} \ov{\{T^j:\:j\in \Z \}}\: \mbox{\em is a
semitopological compact semigroup.} \eeq

Given a factor ${\cal A}\subset{\cal B}$, we have $E(\cdot|{\cal
A})=\Phi_{\mu\otimes_{\cal A}\mu}$. Notice that given $\Phi\in
{\cal J}(T)$, \beq\label{mar5} \Phi\circ E(\cdot|{\cal
A})=E(\cdot|{\cal A})\;\mbox{if and only if}\; \Phi f=f\;\mbox{for
each}\;f\in L^2({\cal A}). \eeq It follows that \beq\label{mar6}
\Phi \circ E(\cdot|{\cal A})=E(\cdot|{\cal A})\; \mbox{if and only
if}\; \eta_\Phi|_{{\cal A}\otimes{\cal A}}=\Delta_{\cal A}. \eeq
Indeed, assume $\eta_\Phi|_{{\cal A}\otimes{\cal A}}= \Delta_{\cal
A}$, fix $f\in L^2({\cal A})$ and let $g\in L^2({\cal A})$ be
arbitrary. Then
$$
\int_X \Phi f(y)g(y)\,d\mu(y)= \int_{X\times
X}f(x)g(y)\,d\eta_\Phi(x,y)= \int_Xfg\,d\mu.
$$
Since $g$ was arbitrary in $L^2({\cal A})$, $\Phi f-f$ is
orthogonal to $L^2({\cal A})$. But we must have $\|\Phi
f\|\leq\|f\|$, and thus $\Phi f=f$. Conversely, if $\Phi f=f$ for
each $f\in L^2({\cal A})$, we get the same equalities for all $f$,
$g\in L^2({\cal A})$, whence $\eta_\Phi|_{{\cal A}\otimes{\cal
A}}=\Delta_{\cal A}$. Now~(\ref{mar6}) follows from~(\ref{mar5}).

In view of (\ref{mar6}) we obtain that \beq\label{mar7} \{\Phi\in
{\cal J}(T):\:\eta_\Phi|_{{\cal A}\otimes{\cal A}}=\Delta_{\cal
A}\}\:\mbox{\em is a compact semitopological semigroup.} \eeq

\begin{Lemma}\label{mar8}
Assume that $\eta\in J(T)$. Then
\begin{multline}
L^2({\cal B}\otimes \{\emptyset,X\},\eta)\cap
L^2(\{\emptyset,X\}\otimes{\cal B},\eta)\\
= \{f\ot1:\:f\in L^2\xbm,\: \|\Phi_\eta f\|=\| f\|\}.
\end{multline}
\end{Lemma}
\begin{proof}
Suppose that $f(x)=g(y)$ for $\eta$--a.e.\ $(x,y)\in X\times X$.
Then we have $\Phi_\eta f=g$ and
$$
\| f\|^2=\int_{X\times X}|f(x)|^2\,d\eta(x,y)= \int_{X\times
X}|g(y)|^2\,d\eta(x,y)=\| g\|^2,$$ so $\|\Phi_\eta f\|=\|f\|$.

On the other hand, take  $f\in L^2\xbm$ satisfying $\|\Phi_\eta
f\|=\|f\|$. Then similarly
$$\int_{X\times X}|f(x)|^2\,d\eta(x,y)=
\int_{X\times X}|\Phi_\eta f(y)|^2\,d\eta(x,y).$$ But, immediately
from the definition, the function $(x,y)\mapsto \Phi_\eta f(y)$ is
the orthogonal projection of $(x,y)\mapsto f(x)$ on the subspace
$L^2(\{\emptyset,X\}\otimes {\cal B})$ of $L^2(X\times X,{\cal
B}\otimes{\cal B},\eta)$. It follows that $f(x)=\Phi_\eta f(y)$
$\eta$-a.e.\ $(x,y)$.
\end{proof}
The $\sigma$-algebra ${\cal B}\otimes \{\emptyset,X\}\cap
\{X,\emptyset\}\otimes{\cal B}$ (modulo $\eta$) can  be seen on
one hand as a factor ${\cal B}_1(\eta)\ot\{\emptyset,X\}$ of
${\cal B}\otimes \{\emptyset,X\}$ and on the other hand as a
factor $\{\emptyset,X\}\ot{\cal B}_2(\eta)$ of
$\{\emptyset,X\}\otimes{\cal B}$. This defines two factors
$\cb_1(\eta)$, $\cb_2(\eta)$ of $\xbm$, the largest factors
identified by the joining $\eta$.

Whenever ${\cal A}\subset{\cal B}$ is a factor of $T$, the
relative product $\mu\otimes_{\cal A}\mu$ is an idempotent in
$J(T)$. The following result states that this is the only way to
obtain idempotents in $J(T)$ (cf.\ Theorem 6.9 in \cite{Gl} where
self-adjoint idempotents of ${\cal J}(T)$ are shown to correspond
to factors).

\begin{Prop}\label{idemp}
Assume that $\eta$ is an idempotent in $J(T)$. Then there exists a
factor ${\cal A}$ of $T$ such that $\eta=\mu\otimes_{\cal A}\mu$.
\end{Prop}
\begin{proof}
Since $\|\Phi_\eta\|=1$, it must be an orthogonal projection and
it is an isometry exactly on its range. Now, in view of
Lemma~\ref{mar8}, $\Phi_\eta$ is an isometry exactly on $L^2({\cal
B}_1(\eta))$. Therefore it is the orthogonal projection onto
$L^2({\cal B}_1(\eta))$, that is $\Phi_\eta=E(\cdot|{\cal
B}_1(\eta))$ and the result follows.
\end{proof}

We can see factors of the form ${\cal B}_1(\eta)$ in a different
way. Indeed, given $\eta\in J(T)$ define
$$
{\cal B}(\eta):=\{A\in{\cal B}:\: \eta\bigl((A\times X)\triangle
(X\times A)\bigr)=0\}.
$$
Then $L^2({\cal B}(\eta))=\{f\in L^2\xbm:\:\Phi_\eta f=f\}$.
Indeed, from the von Neumann theorem for contractions,
$\frac1N\sum_{n=0}^{N-1}\Phi^n_\eta\to {\rm proj}_{{\rm
Fix}(\Phi_\eta)}$ and since the limit is an idempotent and a
Markov operator, it is the orthogonal projection on the
$L^2$--space of a factor.

Recall also that if $W$ is a contraction of a Hilbert space $H$
then so is its adjoint and then $W^\ast Wf=f$ if and only if
$\|Wf\|=\|f\|$. Hence for any $\eta\in J(T)$,
$${\cal B}_1(\eta)={\cal B}(\eta^\ast\circ\eta),$$
 where $\eta^\ast:=\eta_{\Phi^\ast}$.

An automorphism $T$ is said to be {\em rigid} if there exists a
sequence $q_n\to\infty$ such that $T^{q_n}\to$Id in the strong
(or, which here is the same, in the weak) operator topology. We
then say that $(q_n)$ is a {\em rigidity} sequence for $T$.
Suppose now that $\cal A$ is a non-trivial factor of $T$ and
suppose moreover that $(q_n)$ is a rigidity sequence for $T|_{\cal
A}$. Consider
\begin{multline}
I({\cal A}):=\{\Phi\in{\cal J}(T):\:
\eta_\Phi|_{{\cal A}\ot{\cal A}}=\Delta_{\cal A}\\
\:\mbox{and}\;\;\Phi\;\mbox{is a limit point of}\;
{\{T^j:\:j\in\Z\}}\}.
\end{multline}
Note that $I({\cal A})$ is non-empty since any limit Markov
operator of the set $\{T^{q_n}:\:n\geq1\}$ belongs to it. It
follows from~(\ref{mar6}) and~(\ref{mar7}) that $I({\cal A})$ is a
closed subsemigroup of ${\cal J}(T)$, hence a semitopological
compact semigroup. We recall (see {\em e.g.}\ \cite{Gl1}, p.\ 6,
Lemma 2.2) that each compact semitopological semigroup contains an
idempotent. Now, using Proposition~\ref{idemp}, we obtain the
following.

\begin{Prop}\label{idemp1}
Assume that $T$ is an automorphism of $\xbm$ and let ${\cal
A}\subset{\cal B}$ be a non-trivial rigid factor of $T$. Then
there exist a factor ${\cal A}'$ containing ${\cal A}$ and a
rigidity sequence $(q_n)$ for $T|_{{\cal A}'}$ such that
$T^{q_n}\to E(\cdot|{\cal A}')$. \bez
\end{Prop}

\subsection{Canonical factor of a $G$-action
associated to a saturated Borel subgroup}\label{dwajeden} Assume
that $\Lambda$ is a Borel subgroup of $\widehat{G}$. Let us recall
(see \cite{Ho-Me-Pa}) that if $\si,\tau\in M^{+,1}(\widehat{G})$
then $\tau$ {\em sticks to} $\si$ if
$$
\widehat{\si}(g_j)\to 1\; \Longrightarrow\; \widehat{\tau}(g_j)\to
1$$ for any sequence $(g_j)_{j\geq1}$ in $G$ going to infinity.
Following~\cite{Ho-Me-Pa}, one says that $\Lambda$ is {\em
saturated} if for any $\si,\tau\in M^{+,1}(\widehat{G})$ $$
\sigma(\Lambda)=1\;\mbox{and $\tau$ sticks to
$\sigma$}\;\Longrightarrow\; \tau(\Lambda)=1.$$

\begin{Th}[\cite{Ho-Me-Pa}]\label{p4}
Every group $\lf$ is saturated.\bez
\end{Th}
\begin{Remark}\label{sigmafsat}\em As noticed {\em e.g.} in \cite{Le-Pa}, every subgroup
$\Sigma_\varphi$ is also of the form $\Lambda_\psi$, whence
$\Sigma_\varphi$ is also a saturated subgroup.
\end{Remark}
We shall also need the following characterization of saturated
groups.

\begin{Th}[\cite{Ho-Me-Pa}]\label{p3}
A Borel subgroup $\Lambda\subset\widehat{G}$ is saturated if and
only if for any $\tau\in M^+(\widehat{G})$ the indicator function
$1_{\Lambda}$ belongs to the closed convex hull in
$L^1(\widehat{G},\tau)$ of the characters  of $\widehat{G}$.\bez
\end{Th}
The following corollary  describes a dynamical
consequence of Theorem~\ref{p3}~%
\footnote{This consequence of Theorem~\ref{p3} seems to appear for
the first time.}. Given a Borel subset $\Lambda$ of $\widehat{G}$,
we denote by $H_\Lambda$ the spectral subspace corresponding to
$\Lambda$, {\em i.e.\ }the space of those elements in $L^2\ycn$
whose spectral measures are concentrated on $\Lambda$. We denote
by $\tilde{g}$ the character of $\widehat{G}$ associated by
Pontryagin duality to $g\in G$: $\tilde{g}(\chi):=\chi(g)$ for
$\chi\in\widehat{G}$.

\begin{Cor}\label{c1}
Let $\cs$ be an action of $G$ on $\ycn$ and $\Lambda$ be a
saturated subgroup of $\widehat{G}$. Then $H_{\Lambda}= L^2({\cal
A})$ where $\cal A\subset\cc$ is a factor of $\cs$.
\end{Cor}
\begin{proof}
First notice that in Theorem~\ref{p3}, $L^1$-convergence can be
replaced by $L^2$-convergence, in particular
$$
\sum_{k=1}^{N_n}a_k^{(n)}\tilde{g}_k^{(n)}(\chi)\to 1_{\Lambda}\;
\mbox{ in }\;L^2(\widehat{G},\cb(\widehat{G}),\sigma_{\cs})
$$ for some $a^{(n)}_k\geq0$ with
$\sum_{k=1}^{N_k}a^{(n)}_k=1$ and some $g_k^{(n)}\in G$. Then, for
each $\tau\in M^+(\widehat{G})$, $\tau\ll\sigma_{\cs}$ we still
have \beq\label{tom1}
\sum_{k=1}^{N_n}a_k^{(n)}\tilde{g}_k^{(n)}(\chi)\to 1_{\Lambda}\;
\mbox{ in }\;L^2(\widehat{G},\cb(\widehat{G}),\tau). \eeq Consider
$f\in L^2\ycn$. In the canonical representation of $G(f)$, the
function $\sum_{k=1}^{N_n}a_k^{(n)}\tilde{g}_k^{(n)}\in
L^2(\widehat{G},\cb(\widehat{G}),\si_f)$ corresponds to
$\sum_{k=1}^{N_n}a_k^{(n)}S_{g_k^{(n)}}f$ and the subspace
$1_\Lambda\cdot L^2(\widehat{G},\cb(\widehat{G}),\si_f)$
corresponds to $H_\Lambda\cap G(f)$. So, by taking $\tau=\si_f$
in~(\ref{tom1}),
$$\sum_{k=1}^{N_n}a_k^{(n)}S_{g_k^{(n)}}f\to {\rm proj}_{H_\Lambda\cap G(f)}f.$$

Now, since $H_\Lambda$ is a spectral subspace, ${\rm
proj}_{H_\Lambda}f={\rm proj}_{H_\Lambda\cap G(f)}f$. It follows
that the sequence
$\left(\sum_{k=1}^{N_n}a_k^{(n)}S_{g_k^{(n)}}\right)_{n\geq1}$ of
Markov operators of $L^2\ycn$ converges weakly to ${\rm
proj}_{H_\Lambda}$. Therefore, the latter projection is a Markov
operator and the result follows from Proposition~\ref{idemp}.
\end{proof}
The following lemma allows us to localize some eigenvalues of
$\cs$.

\begin{Lemma}\label{l2}
Let $\cs$ be a $G$-action on $\ycn$ and $\Lambda$ be a saturated
subgroup of $\widehat{G}$. Assume that $\si_{\cs}(\Lambda)=0$ and
that $\si_{\cs}(\chi_0\Lambda)>0$ for some $\chi_0\in\widehat{G}$.
Then $\cs$ has an eigenvalue in $\chi_0\Lambda$. More precisely,
there exists exactly one eigenvalue of $\cs$ in $\chi_0\Lambda$
and $H_{\chi_0\Lambda}$ is the eigenspace corresponding to that
eigenvalue.
\end{Lemma}
\begin{proof}
Denote by $\Gamma$ the cyclic group $\{\chi_0^n:\:n\in \Z\}$
considered with the discrete topology. Let $Z$ be the dual group
of $\Gamma$. Hence we obtain a probability space $\zde$, where
${\cal D}=\cb(Z)$ and $\eta$ is the normalized Haar measure on
$Z$. Given $g\in G$ we define $\tilde{g}\in Z$ by
$\tilde{g}(\chi)=\chi(g)$ for each $\chi\in \Gamma$, and
$R_g:\:Z\to Z$ by $R_g(z)=\tilde{g}\cdot z$. In this way we obtain
an ergodic discrete spectrum $G$-action ${\cal R}= (R_g)_{g\in G}$
on $\zde$, whose point spectrum is equal to $\Gamma$. Let us
consider the diagonal $G$-action $\cs\times{\cal R}=(S_g\times
R_g)_{g\in G}$ on $(Y\times Z,{\cal C}\otimes{\cal
D},\nu\otimes\eta)$. The maximal spectral type
$\sigma_{\cs\times{\cal R}}$ of the associated unitary $G$-action
is equal to
$$
\left(\sum_{n\in\Z}\frac1{2^{|n|}}\delta_{\chi_0^n}\right)\ast\si_{\cs}+
\sum_{n\neq0}\frac1{2^{|n|}}\delta_{\chi_0^n}.
$$
By the assumption, we have $\si_{\cs\times{\cal R}}(\Lambda)>0$.
In view of Corollary~\ref{c1}, there exists a non-trivial ${\cal
S}\times{\cal R}$--factor ${\cal A}\subset{\cal C}\otimes{\cal D}$
for which
$$
L^2({\cal A})=\{F\in L^2(Y\times Z,\nu\otimes\eta):\:
\si_{F,\cs\times{\cal R}}\;\mbox{is concentrated on}\;\Lambda\}.
$$
Now fix a non-zero $f\in H_{\chi_0\Lambda}$ and let $h$ be the
eigenfunction $z\mapsto \ov{z(\chi_0)}$ of $\cal R$, corresponding
to the eigenvalue $\ov{\chi}_0$. Then $\sigma_{f\otimes
h,\cs\times{\cal R}}=\delta_{\ov{\chi_0}}\ast\sigma_{f,\cs}$,
hence \beq\label{zima0} \sigma_{f\otimes h,\cs\times{\cal
R}}\,\,\mbox{is concentrated on }\Lambda \eeq and $f\otimes h\in
L^2_0({\cal A})$. Consider the function $|f|^2\otimes h$. First
notice that
$$
|f|^2\otimes h=(f\otimes h)\cdot(\ov{f}\otimes 1),
$$
so the function $|f|^2\otimes h$ is measurable with respect to
${\cal A}\vee({\cal C}\otimes \{\emptyset,Z\})$.

The two $G$-actions $(\cs\times{\cal R})|_{\cal A}$ and $\cs$ are
spectrally disjoint since, by assumption, $\si_{\cs}(\Lambda)=0$.
Hence, they are disjoint. In particular, $|f|^2\otimes h$ is in
$L^2(Y\times Z,\nu\otimes\eta)$ and \beq\label{zima1}
\sigma_{|f|^2\otimes h,\cs\times{\cal R}}= \sigma_{f\otimes
h,(\cs\times{\cal R})|_{\cal A}}\ast\sigma_{\ov{f},\cs}. \eeq At
the same time, since $\int_{Y}|f|^2\,d\nu>0$, we have
$\delta_{1}\ll \sigma_{|f|^2,\cs}$ and therefore
$$
\sigma_{|f|^2\otimes h,\cs\times{\cal R}}=
\sigma_{|f|^2,\cs}\ast\sigma_{h,{\cal R}}\gg\delta_{1}\ast
\delta_{\ov{\chi}_0}=\delta_{\ov{\chi}_0}.
$$

It now follows directly from (\ref{zima1}) that
$\sigma_{\ov{f},\cs}$ is not a continuous measure. More precisely,
in view of (\ref{zima0}), $\sigma_{\ov{f},\cs}$ must have a point
mass at some $\chi\in\widehat{G}$ such that $\ov{\chi}_0\in\chi
\Lambda$, and $f$ cannot be orthogonal to the subspace of
eigenfunctions corresponding to the eigenvalues of $\cs$ in
$\chi_0\Lambda$. Since $f$ is an arbitrary element of
$H_{\chi_0\Lambda}$, the space $H_{\chi_0\Lambda}$ consists only
of eigenfunctions. Finally, since $\sigma_{\cs}(\Lambda)=0$, no
two different eigenvalues of $\cs$ can be in the coset
$\chi_0\Lambda$ and the proof is complete.
\end{proof}

\begin{Remark} \em
Note that $\sigma_{\cs}(\Lambda)=0$ implies that $\cs$ is ergodic.
It follows that under the assumptions of the above lemma,
$H_{\chi_0\Lambda}$ is moreover one-dimensional.
\end{Remark}

Although, the result below (Proposition~\ref{rfmp}) will not be
used in what follows, we bring it up, as it is another sample of
applications of saturated groups in (non-singular) ergodic theory.

Assume that $T$ is a non-singular ergodic automorphism of a
standard probability space $\xbm$ and that $S:\ycn\to\ycn$ is
another non-singular automorphism. Let $\pi:\xbm\to\ycn$ settle a
homomorphism between $T$ and $S$. Following \cite{Da-Le},
$(S,\pi)$ is called a {\em relatively finite measure-preserving}
(rfmp) factor of $T$ if $\frac{d\mu\circ T}{d\mu}$ is
$\pi^{-1}({\cal C})$-measurable.

\begin{Prop}\label{rfmp} Assume that $T$ is a nonsingular ergodic
automorphism and $S$ is an rfmp factor of it. Let $R$ be a weakly
mixing probability preserving automorphism of a standard
probability Borel space $\zde$. Assume that $R\times S$ is
ergodic. Then $R\times T$ is also ergodic \footnote{Usually, to
ensure ergodicity one has to assume mild mixing of $R$, a property
stronger than weak mixing. As a matter of fact, $R$ is mildly
mixing if and only if $R\times T$ is ergodic for each non-singular
ergodic $T$, see \cite{Fu-We}. By definition (\cite{Fu-We}), mild
mixing means that $R$ has no non-trivial rigid factors.}.
\end{Prop}
\begin{proof} We need to show that $\sigma_R(e(T))=0$, where
$e(T)$ stands for the group of $L^\infty$ eigenvalues ($c\in
e(T)\subset\T$ if for some $f\in L^\infty\xbm$, $f\circ T=c\cdot
f$), see {\em e.g.}\ \cite{Qu}. In view of Theorem~2 in
\cite{Au-Le}, $e(T)$ is the union of  countably many cosets
$c\cdot e(S)$, $e(T)=\bigcup_{i=1}^\infty c_i\cdot e(S)$.  On the
other hand, $\sigma_R(e(S))$=0 and $e(S)$ is saturated (see
\cite{Ho-Me-Pa}). Since $R$ is weakly mixing, in view of
Lemma~\ref{l2}, $\sigma_R(c\cdot e(S))=0$ for each $c\in\T$.
Therefore, $\sigma_R(e(T))=0$ and the result follows.
\end{proof}

Following \cite{Da-Le}, given a non-singular automorphism $S$ of
$\ycn$, to obtain $T$ and $\pi$ so that $(S,\pi)$ is an rfmp
factor of $T$ we must take any ergodic skew product $T=S_\Theta$
on $\xbm=\ycn\ot(X',\cb',\mu')$ where $(X',\cb',\mu')$ is another
probability standard Borel space and $\Theta:X\to {\rm
Aut}(X',\cb',\mu')$ is measurable.

\subsection{Lemma on mixing times of weighted unitary operators}\label{weightedops}
In the study of mixing properties of automorphism of the form
$\tfs$ unitary operators $V_\xi$ defined below will play a crucial
role.

Let $T$ be an automorphism of $\xbm$. Let $\xi:X\to\T$ be a
cocycle. We define a unitary operator $V_\xi$ on $L^2\xbm$ by
setting
$$
V_\xi(f)(x)=\xi(x)\cdot f(Tx)
$$
for each $f\in L^2\xbm$. A sequence $(n_i)\subset\N$,
$n_i\to\infty$ is said to be a {\em mixing sequence} for $V_\xi$
if $V_\xi^{n_i}\to 0$ in the weak operator topology (while $(n_i)$
is a mixing sequence for $T$ if $T^{n_i}$ restricted to
$L^2_0\xbm$ goes to $0$, that is, $T^{n_i}\to \Phi_{\mu\ot\mu}$ in
${\cal J}(T)$).

Denote by $T_\xi$ the Anzai skew product corresponding to $T$ and
$\xi$, {\em i.e.} the automorphism of  $(X\times \T,{\cal
B}\otimes{\cal B}(\T),\mu\otimes\lambda)$ given by
$T_\xi(x,z)=(Tx,\xi(x)z)$, where $\lambda$ stands for the Lebesgue
measure of the circle.

We assume now that $T$ is ergodic.

\begin{Prop}\label{lifttomix}
Assume that $T^{n_i}\to \Phi$ in ${\cal J}(T)$, where
$\Phi=\Phi_\rho\in{\cal J}^e(T)$. Suppose that the $(T\times
T,\rho)$-cocycle $\xi\otimes\ov{\xi}$ is not a coboundary. Then
$(n_i)$ is a mixing sequence for $V_\xi$.
\end{Prop}
\begin{proof}
We can moreover assume that $(T_{\xi})^{n_i}\to
 \widetilde{\Phi}=\Phi_{\tilde{\rho}}$ in ${\cal J}(T_\xi)$.
Given $f_1$, $f_2\in L^2\xbm$, we define  $F_i\in
L^2(X\times\T,\mu\ot\la)$ by $F_i(x,z)=f_i(x)z$, $i=1,2$. Set also
$J(x_1,z_1,x_2,z_2)=z_1\ov{z}_2$, so $J\in
L^2((X\times\T)\times(X\times\T),\tilde{\rho})$. Moreover let,
with some abuse of notation, $H$ denote
$E^{\tilde{\rho}}(J|\cb\ot\cb)$. We have
\begin{equation*}
\begin{split}
\int_{X}(V_\xi)^{n_i}f_1\cdot\ov{f}_2\,d\mu
&=\int_{X}\xi^{(n_i)}\cdot f_1\circ T^{n_i}\cdot \ov{f}_2\,d\mu\\
&=\int_{X\times\T}F_1\circ T_\xi^{n_i}\cdot\ov{F}_2\,d\mu\,d\la\\
&\longrightarrow
\int_{(X\times\T)\times(X\times\T)}\!\!\!\!F_1(x_1,z_1)\ov{F_2(x_2,z_2)}\,
d\widetilde{\rho}((x_1,z_1),(x_2,z_2))\\
&=\int_{(X\times\T)\times(X\times\T)}\!\!\!f_1(x_1)f_2(x_2)z_1\ov{z}_2\,
d\widetilde{\rho}((x_1,z_1),(x_2,z_2))\\
&=\int_{X\times X}f_1\ot f_2\cdot H\,d\rho.
\end{split}
\end{equation*}
We claim now that, for $\rho$-a.a.\ $(x_1,x_2)\in X\times X$,
\beq\label{lm1} \xi(x_1)\ov{\xi(x_2)}H(x_1,x_2)=H(Tx_1,Tx_2). \eeq
Indeed, $J\circ (T_\xi\times T_\xi)=(\xi\ot\ov{\xi})\cdot J$, so
\begin{equation*}
\begin{split}
E^{\widetilde{\rho}}(J|\cb\otimes \cb)\circ (T\times T)
&=E^{\widetilde{\rho}}(J\circ (T_\xi\times T_\xi)|\cb\otimes\cb)\\
&=\xi\ot\ov{\xi}\cdot E^{\widetilde{\rho}}(J|\cb\otimes \cb)
\end{split}
\end{equation*}
and (\ref{lm1}) follows.

Now, ergodicity of $\rho$ implies that $H$ is of constant modulus.
If $H\neq 0$ then from (\ref{lm1}) it follows that
$\xi\ot\ov{\xi}$ is a $(T\times T,\rho)$-coboundary. Otherwise
$\int_X(V_\xi)^{n_i}f_1\cdot \ov{f}_2\,d\mu\to0$ for all $f_1$,
 $f_2\in L^2\xbm$, so $(n_i)$ is a
mixing sequence for $V_{\xi}$.
\end{proof}

\begin{Cor}\label{lm2}
If $T$ is weakly mixing and $(n_i)$ is a mixing sequence for $T$,
then $(n_i)$ is also a mixing sequence for $V_\xi$ whenever $\xi$
is not a quasi-coboundary.
\end{Cor}
\begin{proof}
We have then $T^{n_i}\to\Phi_{\mu\ot\mu}$ in ${\cal J}(T)$. As $T$
is weakly mixing, $\mu\ot\mu\in{\cal J}^e(T)$, and it is
well-known that $\xi\ot\ov{\xi}$ is a $(T\times
T,\mu\ot\mu)$-coboundary if and only if $\xi$ is a
quasi-coboundary (see {\em e.g.}\ \cite{Le-Pa}, Appendix).
\end{proof}

\subsection{Recurrent cocycles with values in Abelian Polish groups}
The remarks below about recurrent cocycles with values in Polish
Abelian groups are taken directly from the theory of cocycles
taking values in LCA groups \cite{Sch}, we give proofs only for
sake of completeness.

Let $T$ be an ergodic automorphism of $\xbm$. Assume that $A$ is
an Abelian\footnote{We keep going to use multiplicative notation.}
Polish group. Let $\va:X\to A$ be a cocycle. The cocycle $\va$ is
said to be {\em recurrent} if for each $\vep>0$, $B\in{\cal B}$ of
positive measure and each neighbourhood $V$ of $1$, there exists a
positive integer $N$ such that \beq\label{rec1} \mu(B\cap
T^{-N}B\cap[\va^{(N)}\in V])>0. \eeq Suppose that $\psi:X\to A$ is
another cocycle. Then we have the following fact: \beq\label{rec2}
\mbox{\em if $\va$ and $\psi$ are cohomologous and $\va$ is
recurrent, then so is $\psi$}. \eeq Indeed, assume that
$\psi=\va\cdot f\cdot (f\circ T)^{-1}$ for a measurable $f:X\to
A$. Take a set $B\in{\cal B}$ of positive measure. Fix a
neighbourhood $V$ of $1$, and then another neighbourhood $W$ of
$1$ so that $W\cdot W\subset V$. Using measurablity  of $f$, we
can find a measurable subset $B_1\subset B$ of positive measure
such that $f(x)\cdot f(y)^{-1}\in W$ whenever $x$, $y\in B_1$.
Then, for every $N\geq 1$,
$$
B_1\cap T^{-N}B_1\cap[\va^{(N)}\in W]\subset B\cap
T^{-N}B\cap[\psi^{(N)}\in V]
$$
and (\ref{rec2}) follows.

Assume now that $a\in A$ and let $\widetilde{a}$ denote the
corresponding constant cocycle: $\widetilde{a}(x)=a$. Then the
following fact holds: \beq\label{rec3}
\begin{split}
&\mbox{\em $\tilde{a}$ is recurrent if and only if}\\
&\mbox{\em there exists $n_j\to\infty$ such that $a^{n_j}\to1$ in
$A$}.
\end{split}
\eeq Indeed, fix a neighbourhood $V$ of $1$ and apply (\ref{rec1})
with $B=X$ to obtain that $\mu([\tilde{a}^{(N)}\in V])>0$ for some
positive integer $N$. It follows that $a^N\in V$. Letting
$V\to\{1\}$, either $N=N(a,V)\to\infty$ and we are done, or $N$
stays bounded. In the latter case we have $a^N=1$ for some
$N\geq1$ and by taking multiples of this $N$,~(\ref{rec3}) also
follows.

On the other hand, if $a^{n_j}\to 1$ then the sequence of
differences $n_i-n_j$ is a sequence universally good for the
Poincar\'e recurrence and clearly $a^{n_i-n_j}\to1$ when
$i,j\to\infty$. Therefore $\widetilde{a}$ is recurrent.

\section{Lifting mixing properties to Rokhlin cocycle
extensions}\label{TRZY} In this section we will present a
systematic study of mixing properties of automorphisms of the form
$\tfs$. Throughout $T$ is assumed to be an automorphism of $\xbm$,
$\va:X\to G$ a cocycle and $\cs=(S_g)_{g\in G}$ a $G$-action
acting on $\ycn$.

\subsection{Maximal spectral type of \protect$\tfs$}
Let $\{f_n\}_{n\geq0}$ and $\{g_n\}_{n\geq0}$ be  orthonormal
bases in $L^2\xbm$ and $L^2\ycn$ respectively, where $f_0=g_0=1$.
For the maximal spectral type $\stfs$ of $\tfs$ on $L^2_0(X\times
Y,\mu\otimes\nu)$, we take%
\footnote{Up to some abuse of vocabulary, we take as
$\sigma_{\tfs}$ any spectral measure realizing the maximal
spectral type.} \beq\label{e1}
\stfs=\sum_{(m,n)\neq(0,0)}2^{-(m+n)}\si_{f_n\otimes g_m,\tfs}.
\eeq According to the notation of section \ref{weightedops}, given
$\chi\in\widehat{G}$, we denote by $\vhf$ the unitary operator on
$L^2\xbm$ which acts by the formula
$$
(\vhf f)(x)=\chi(\va(x))f(Tx).
$$
Its maximal spectral type, on $L^2\xbm$, is equal to
$$
\si_{\vhf}=\sum_{n\geq0}\frac1{2^n}\;\si_{f_n,\vhf}.
$$
Notice also that the maximal spectral type of $\cs$ on $L^2_0\ycn$
is given by
$$
\si_{\cs}=\sum_{m\geq1}\frac1{2^m}\;\si_{g_m,\cs}
$$
and $\si_T$, the maximal spectral type of $T$ on $L^2_0\xbm$, is
equal to $\sum_{n\geq1}\frac1{2^n}\,\si_{f_n,T}$.

\begin{Lemma}\label{l1}
We have
$\stfs=\si_T+\int_{\widehat{G}}\sigma_{\vhf}\,d\sigma_{\cs}(\chi)$.
Moreover
$$
\sigma_{T_{\varphi,{\cal S}}|_{L^2(X\times Y,\mu\otimes\nu)\ominus
L^2(X,\mu)\ot1_Y}}=\int_{\widehat{G}}\sigma_{\vhf}\,d\sigma_{\cs}(\chi).
$$
\end{Lemma}

\begin{proof} Firstly, we calculate the spectral measure of
$f\otimes g$ for $f\in L^2\xbm$, $g\in L^2\ycn$. For each $k\in
\Z$ we have
\begin{equation*}
\begin{split}
\widehat{\si}_{f\otimes g,\tfs}(k)
&=\int_{X\times Y}(f\otimes g)\circ(\tfs)^k\cdot\ov{f\otimes g}\,d(\mu\otimes \nu)\\
&=\int_X f(T^kx)\ov{f(x)}\left(\int_Y
g(S_{\va^{(k)}(x)}(y))\ov{g(y)}
\,d\nu(y)\right)\,d\mu(x)\\
&=\int_X f(T^kx)\ov{f(x)}\left(\int_{\widehat{G}}
\chi(\va^{(k)}(x))
\,d\si_{g,\cs}(\chi)\right)\,d\mu(x)\\
&=\int_{\widehat{G}}\left( \int_X
\chi(\va^{(k)}(x))f(T^kx)\ov{f(x)}
\,d\mu(x)\right)\,d\si_{g,\cs}(\chi).\\
&=\int_{\widehat{G}}\widehat{\si}_{f,\vhf}(k)\,d\si_{g,\cs}(\chi).
\end{split}
\end{equation*}
It follows that \beq\label{sifotg} \si_{f\ot g,\tfs}=
\int_{\widehat{G}}\si_{f,\vhf}\,d\si_{g,\cs}(\chi). \eeq
Therefore, in view of (\ref{e1}) and (\ref{sifotg})
\begin{equation*}
\begin{split}
\si_{\tfs}&=\sum_{(n,m)\neq(0,0)}\frac1{2^{n+m}}\si_{f_n\ot g_m,\tfs}\\
&=\int_{\widehat{G}}\sum_{(n,m)\neq(0,0)}\frac1{2^{n+m}}
\si_{f_n,\vhf}\,d\si_{g_m,\cs}(\chi)\\
&=\sum_{m\geq1}\frac1{2^m}\int_{\widehat{G}}\sum_{n\geq0}\frac1{2^n}
\si_{f_n,\vhf}\,d\si_{g_m,\cs}(\chi)
+\int_{\widehat{G}}\sum_{n\geq1}
\frac1{2^n}\si_{f_n,\vhf}\,d\si_{g_0,\cs}(\chi)\\
&=
\sum_{m\geq1}\frac1{2^m}\int_{\widehat{G}}\si_{\vhf}\,d\si_{g_m,\cs}(\chi)
+\sum_{n\geq1}\frac1{2^n}\si_{f_n,T}\\
&=\int_{\widehat{G}}\si_{\vhf}\,d\si_{\cs}(\chi)+\si_T.
\end{split}
\end{equation*}
The result immediately follows.
\end{proof}

\subsection{Maximal spectral
type of \protect$\tfs$ on subspaces of the form
\protect$L^2\xbm\otimes G(g)$}

Assume now that $g\in L^2_0\ycn$. Recall that by $G(g)$ we denote
the cyclic space generated by $g$, {\em i.e.}
$$
G(g)=\ov{\mbox{span}}\{g\circ S_h:\:h\in G\}.
$$
The space $L^2\xbm\otimes G(g)$ is $\tfs$-invariant. Indeed, we
can naturally identify $L^2(X\times Y, \cb\otimes\cc,
\mu\otimes\nu)$ with the space $L^2\bigl(X, \cb,
\mu;\,L^2\ycn\bigr)$ of square-integrable $L^2\ycn$-valued
functions on $\xbm$, and then $L^2\xbm\otimes G(g)$ becomes the
subspace of functions taking $\mu$-a.e.\ their values in $G(g)$.
Now, if $F(x,\cdot)\in G(g)$, then
$(F\circ\tfs)(x,\cdot)=F(Tx,\cdot)\circ S_{\va(x)}\in G(g)$.

For each $h\in G$, as $\si_{g\circ S_h,\cs}=\si_{g,\cs}$, we have
from (\ref{sifotg})
$$\sigma_{f\otimes g,\tfs}=\sigma_{f\otimes(g\circ S_h),\tfs}.$$
Let $\{f_n\}_{n\geq0}$ be an orthonormal base in $L^2\xbm$ with
$f_0=1$. It is then clear that \beq\label{zima2}
\sigma_{\tfs|_{L^2(X,\mu)\otimes
G(g)}}=\sum_{n\geq0}2^{-n}\sigma_{f_n\otimes g,\tfs}. \eeq
Therefore, by the proof of Lemma~\ref{l1}, we obtain the
following.

\begin{Lemma}\label{zima3}
We have
$$
\sigma_{\tfs|_{L^2(X,\mu)\otimes G(g)}}
=\int_{\widehat{G}}\sigma_{\vhf}\,d\sigma_{g,\cs}(\chi).
$$
\bez
\end{Lemma}

\subsection{Ergodicity of $\tfs$}
We assume here that $T$ is ergodic. Let us first notice that, then
\beq\label{e2} \chi\in\lf\;\;\mbox{\em if and only
if}\;\;\si_{\vhf}(\{1\})>0 \eeq or, more generally, that

\begin{equation}\label{e3}
\begin{split}
&\mbox{\em The cocycle}\;\chi\circ\va\;\mbox{\em is cohomologous to}\  e^{2\pi it}\\
&\mbox{\em if and only if }\si_{\vhf}(\{e^{2\pi it}\})>0.
\end{split}
\end{equation}
Indeed, $\si_{\vhf}(\{e^{2\pi it}\})>0$ if and only if $e^{2\pi i
t}$ is an eigenvalue of $\vhf$ and any eigenfunction corresponding
to this eigenvalue will have constant modulus and so, up to
normalization, be a transfer function $j$ in the cohomology
equation $\chi\circ\va=e^{2\pi it}\cdot j/j\circ T$.

The result below has already been proved in \cite{Le-Le}. We give
however a shorter proof.

\begin{Prop}[\cite{Le-Le}]\label{p1}
$\tfs$ is ergodic if and only if $T$ is ergodic and\goodbreak
$\si_{\cs}(\lf)=0$.
\end{Prop}
\begin{proof} It is clearly necessary that $T$ be ergodic.
Then, by Lemma~\ref{l1}, $\si_{\tfs}(\{1\})=0$ if and only  if
$\si_{\vhf}(\{1\})=0$ for $\si_{\cs}$-a.e.\ $\chi\in\widehat{G}$
and therefore, in view of~(\ref{e2}), if and only if
$\si_{\cs}(\lf)=0$.
\end{proof}

\begin{Remark}\em Let us notice that $\si_{\cs}(\lf)=0$ implies
that $\si_{\cs}(\{1\})=0$. Indeed a necessary condition for
ergodicity of $\tfs$ is the ergodicity property of $\cs$ itself.
\end{Remark}

\subsection{Eigenvalues of \protect$\tfs$}
Assume now that $\tfs$ is ergodic. We will determine its
eigenvalues (and eigenfunctions). Let us fix $t\in[0,1)$ and set
$$
A_t=\{\chi\in\widehat{G}:\: \chi\circ\va\;\mbox{is cohomologous
to}\;e^{2\pi it}\}.
$$
Notice that $A_t\subset\svf$ and that if $\chi\in A_t$ and
$\chi_1\in\lf$ then $\chi\chi_1\in A_t$. Moreover, if
$\chi_1,\chi_2\in A_t$ belong to $A_t$ then $\chi_1\ov{\chi}_2\in
\lf$. It follows that $A_t$ is a coset of $\lf$.

Suppose that $e^{2\pi it}$ is an eigenvalue of $\tfs$, {\em i.e.}\
$\si_{\tfs}(\{e^{2\pi it}\})>0$ and let $F$ be a corresponding
eigenfunction. We shall assume that $F$ is not a function of $x$
alone (otherwise $F$ is an eigenfunction of $T$ and the result
below is trivial). Then there exists $g\in L^2_0\ycn$ such that
$F$ is not orthogonal to $L^2\xbm\otimes G(g)$. Since the spectral
measure of $F$ is the Dirac measure at $e^{2\pi it}$, it follows
from Lemma~\ref{zima3} that
$$
\int_{\widehat{G}}\sigma_{V_{\chi\circ\varphi}}\,d\sigma_{g,\cs}(\chi)
(\{e^{2\pi it}\})>0.
$$
The latter occurs if and only if
$\si_{g,\cs}(\{\chi\in\widehat{G}:
\:\si_{V_{\chi\circ\varphi}}(\{e^{2\pi it}\})>0\})>0$, that is, by
(\ref{e3}), if and only if $\si_{g,\cs}(A_t)>0$. Now  $A_t$ is a
coset $\chi_0\lf$ of $\lf$ and there exists then a non-zero
$g_1\in G(g)\cap H_{\chi_0\lf}$. According to Lemma~\ref{l2} (and
Theorem~\ref{p4}), since $\tfs$ is ergodic and thus
$\si_\cs(\lf)=0$, it follows that $g_1$ is an eigenfunction
corresponding to an eigenvalue $\chi\in A_t$.

Let now $f$ be a measurable function of modulus 1 satisfying
$\chi\circ \varphi\cdot f\circ T=e^{2\pi it}\cdot f$ $\mu$-a.e. As
$g_1\circ S_{\va(x)}=\chi(\va(x))\cdot g_1$, we get
$$
(f\otimes g_1)\circ \tfs=(\chi\circ\va\cdot f\circ T)\otimes g_1
=e^{2\pi it}\cdot(f\otimes g_1).
$$
So, $f\otimes g_1$ is an eigenfunction of $\tfs$ corresponding to
$e^{2\pi it}$ and, since $\tfs$ is ergodic, $F$ can be different
from $f\otimes g_1$ only by a multiplicative constant. Therefore
we have proved the following.

\begin{Prop}\label{zima4}
Assume that $\tfs$ is ergodic. Then the  eigenfunctions of $\tfs$
are the functions of  the form $f\otimes g$, where
$\chi\circ\varphi=e^{2\pi it}\cdot f/f\circ T$, $\chi$ is an
eigenvalue of $\cs$ and $g$ is an eigenfunction corresponding to
$\chi$. In particular, $e^{2\pi it}$ ($t\in[0,1)$) is an
eigenvalue of $\tfs$ if and only if there exists an eigenvalue of
$\cs$ in $A_t$. \bez
\end{Prop}

\subsection{Weak mixing and relative weak mixing}
A characterization of the weak mixing property for $\tfs$ is a
direct corollary of Proposition~\ref{zima4}.

\begin{Cor}\label{c7}
$\tfs$ is weakly mixing if and only if it is ergodic, $T$ is
weakly mixing and $\cs$ has no eigenvalues in $\svf$.\bez
\end{Cor}

\begin{Remark}\em
Notice that this corollary generalizes the well-known criterion
for weak mixing property of Abelian compact group extensions.
\end{Remark}

Let us pass to a characterization of the relative weak mixing
property. We still assume that $\tfs$ is ergodic.

Let us first notice that the relative product of $\tfs$ with
itself over the factor $T$ is isomorphic to
$T_{\va,\cs\times\cs}$, where $\cs\times\cs$ stands for the
diagonal action $g\mapsto S_g\times S_g$ of $G$ on $(Y\times
Y,\cc\ot\cc,\nu\ot\nu)$. So $\tfs$ is relatively weakly mixing
over $T$ if and only if $T_{\va,\cs\times\cs}$ is ergodic.

Since $\si_{\cs\times\cs}=\si_\cs+\si_{\cs}\ast\si_{\cs}$, it
follows from Proposition~\ref{p1} that $\tfs$ is relatively weakly
mixing over $T$ if and only if
$\si_{\cs}(\lf)+\si_{\cs}\ast\si_{\cs}(\lf)=0$. The latter
statement is equivalent to saying that $\si_{\cs}(\chi\lf)=0$ for
each $\chi\in\widehat{G}$. This has already been proved in
\cite{Le-Le} but now we have Lemma~\ref{l2} at our disposal which
finally improves and clarifies the result: Since $\tfs$ is
ergodic, we have $\si_{\cs}(\lf)=0$, and $\si_{\cs}(\chi_0\lf)>0$
for some $\chi_0$ if and only if $\cs$ has an eigenvalue.

\begin{Prop}\label{p5}
$\tfs$ is relatively weakly mixing over $T$ if and only if it is
ergodic and $\cs$ is weakly mixing.\bez
\end{Prop}

\subsection{Relative Kronecker factor of \protect$\tfs$ over
\protect$T$} Denote by $\ck(\cs)\subset{\cal C}$ the Kronecker
factor of $\cs$, {\em i.e.}\ the factor generated by the
eigenfunctions of the unitary action of $\cs$. If $g,h$ are
eigenfunctions of $\cs$ (corresponding to $\chi$ and $\chi'$
respectively) then
$$
(g\otimes h)\circ T_{\va,\cs\times\cs}(x,y,y')=
\chi(\va(x))\chi'(\va(x))\cdot (g\otimes h)(y,y').
$$
It follows that ${\cal B}\otimes \ck(\cs)$ is contained in the
relative Kronecker factor of $\tfs$ over $T$ (cf.~(\ref{def1})
and~(\ref{def2}) for $n=1$). In fact, we have the following.

\begin{Prop}\label{rKron1}
Assume that $\tfs$ is ergodic. The relative Kronecker factor of
$\tfs$ over $T$ is equal to $\cb\otimes\ck(\cs)$.
\end{Prop}
\begin{proof}
Assume that $F\in L^2(X\times Y\times Y,\mu\otimes\nu\otimes\nu)$
is a $T_{\va,\cs\times\cs}$-invariant function. Take $g,h\in
L^2\ycn$ and suppose that $F$ is not orthogonal to $L^2\xbm\otimes
G(g\otimes h)$. Then, proceeding as in the proof of
Proposition~\ref{zima4}, we obtain that $\sigma_{g\otimes
h,\cs\times\cs}(\Lambda_{\va})=
\sigma_{g,\cs}\ast\sigma_{h,\cs}(\Lambda_{\varphi})>0$. Therefore
$\sigma_{g,\cs}(\chi\Lambda_\varphi)>0$ for some
$\chi\in\widehat{G}\setminus\{1\}$ (and the same holds for $h$).
By Lemma~\ref{l2}, remembering that $\tfs$ is ergodic, $g$ is not
orthogonal to an eigenfunction of $\cs$ from $H_{\chi\Lambda}$. It
follows that if $\{g_i\}_{i\geq0}$ stands for an orthonormal base
of $L^2(\ck(\cs))$, where each $g_i$ is an eigenfunction
corresponding to $\chi_i$, $i\geq0$, then
$$
F(x,y,y')=\sum_{i,j\geq0}a_{ij}(x)g_i(y)g_j(y'),
$$
where $\chi_i\cdot\chi_j\in\Lambda_\varphi$, whenever
$a_{ij}\neq0$ (in fact $\chi_j=\ov{\chi_i}$ since there is at most
one eigenvalue in a coset $\chi\lf$). Fix any function
$J=J(x,y')\in L^2(X\times Y,\mu\otimes\nu)$. Then, for each
$i,j\geq0$ the function given by
$$
\left((a_{ij}\otimes g_i\otimes g_j)\ast J\right)(x,y):= \int_Y
a_{ij}(x)g_i(y)g_j(y')J(x,y')\,d\nu(y')$$ is of the form
$A(x)g_i(y)$, so it is measurable with respect to $\cb\otimes
\ck(\cs)$. The result follows then directly from the description
of the relative Kronecker factor given in \cite{Fu3},
Theorem~6.13.
\end{proof}

\begin{Remark}\em Proposition~\ref{rKron1} yields another proof of
the result about eigenfunctions of $\tfs$ when $\tfs$ is ergodic.
Indeed, eigenfunctions are measurable with respect to the relative
Kronecker factor. If, as before, $\{g_i\}$ stands for an
orthonormal base of eigenfunctions in $L^2(\ck(\cs))$ and $F\circ
\tfs=c\cdot F$, then
$$
F(x,y)=\sum_{i\geq0}a_i(x)g_i(y),\;F\circ\tfs(x,y)=\sum_{i\geq0}a_i(Tx)
\chi_i(\va(x))g_i(y),$$ so $c\cdot
a_i(x)=a_i(Tx)\cdot\chi_i(\va(x))$ for each $i\geq0$.
\end{Remark}

\subsection{Regularity of Rokhlin cocycles}\label{tensorowa}
When the group $\lf$ of a cocycle $\va:X\to G$ is not trivial
then, obviously, the skew product $T_\va:(X\times G,\mu\ot
\la_G)\to (X\times G,\mu\ot\la_G)$,
$T_\varphi(x,g)=(Tx,\va(x)\cdot g)$, is not ergodic, but $\va$
need not be a coboundary. We will show however in this section
that on the level of Rokhlin cocycles, that is when considering
the cocycle $x\mapsto S_{\va(x)}$, it must be a coboundary as soon
as $\sigma_{\cs}$ is concentrated on $\lf$\footnote{When $\lf$ is
uncountable then there is always a weakly mixing Gaussian action
$\cs$ such that $\sigma_{\cs}$ is concentrated on $\lf$.}. In the
general case, we denote by $\ca_{\lf}$ be the $\cs$-factor
corresponding to $\lf$ according to Corollary~\ref{c1}, {\em i.e.}
$L^2(\ca_{\lf})=H_{\lf}$, and we will show that $x\mapsto
S_{\va(x)}|_{\ca_{\lf}}$ is a coboundary.

We show firstly that, when $T$ is ergodic, $\cb\otimes\ca_{\lf}$
contains the factor of $\tfs$-invariant sets.

\begin{Lemma}\label{tens3}
Assume that $T$ is ergodic. Every $\tfs$-invariant function $F$ in
$L^2(X\times Y,\cb\otimes\cc, \mu\otimes\nu)$ is
$\cb\otimes\ca_{\lf}$-measurable.
\end{Lemma}
\begin{proof}
Given any $g\in L^2_0\ycn$, the projection of $F$ on the
$\tfs$-invariant subspace  $L^2\xbm\otimes G(g)$ is still
$\tfs$-invariant. If this projection is non-zero, the maximal
spectral type of $\tfs$ on $L^2\xbm\otimes G(g)$ must have an atom
at~1. Then, by Lemma~\ref{zima3},
$\si_{g,\cs}(\{\chi\in\widehat{G}:\:\sigma_{\vhf}(\{1\})>0\}>0$
and so, since $T$ is ergodic, $\sigma_{g,\cs}(\lf)>0$ in view of
(\ref{e2}). Since $g$ was arbitrary, it follows $F\in
L^2\xbm\otimes H_{\lf}= L^2(X\times Y, \cb\otimes\ca_{\lf},
\mu\otimes\nu)$.
\end{proof}

We give now a short description of the action of $\tfs$ on
$L^2\xbm\otimes G(g)$ for a non-zero $g\in L^2\xbm$, which will
shed some light on the proof of Proposition~\ref{tens6} below. We
identify naturally $L^2\xbm\otimes G(g)$ to $L^2(X, \cb,\mu;\;
G(g))$. We may furthermore replace $G(g)$ by $L^2(\widehat{G},
\cb(\widehat{G}), \si_{g,\cs})$ through the canonical spectral
isomorphism, which we denote by $I$. We shall determine the
unitary operator $V_\va$ on $L^2(X, \cb, \mu;\;
L^2(\widehat{G},\cb(\widehat{G}),\sigma_{g,\cs}))
=L^2(X\times\widehat{G},\cb\otimes\cb(\widehat{G}),
\mu\otimes\sigma_{g,\cs})$ which corresponds to $\tfs$ acting on
$L^2\xbm\otimes G(g)$.

Let $\widetilde{F}\in
L^2(X\times\widehat{G},\cb\otimes\cb(\widehat{G}),
\mu\otimes\sigma_{g,\cs})$ correspond to $F\in L^2\xbm\otimes
G(g)$ - {\em i.e.\ }$\widetilde{F}(x,\cdot)=I(F(x,\cdot))$ for
$\mu$-a.e.\ $x$. We have, for $\mu$-a.e.\ $x$,
$$
(V_\varphi\widetilde{F})(x,\cdot)=I(F(Tx,S_{\va(x)}(\cdot))).
$$
Now, if $h\in G$, the image of $g\circ S_h$ under $I$ is $\tilde
h$, where $\tilde h(\chi)=\chi(h)$ for $\chi\in \widehat{G}$. If
we take $F$ of the form $F=f\otimes g\circ S_h$ then
$\widetilde{F}=f\otimes\widetilde{h}$ and
$$
V_\varphi(f\otimes\widetilde{h})(x,\cdot) =f(Tx)\cdot I(g\circ
S_{\va(x)\cdot h}) =\chi(\va(x))\cdot
(f\otimes\tilde{h})(Tx,\cdot).
$$
It follows that for each $\widetilde{F}\in
L^2(X\times\widehat{G},\mu\otimes\sigma_{g,\cs})$,
$\mu\otimes\sigma_{g,\cs}$-a.e.\ $(x,\chi)\in X\times\widehat{G}$,
\beq\label{tens2} (V_\varphi
\widetilde{F})(x,\chi)=\chi(\va(x))\widetilde{F}(Tx,\chi). \eeq

We come to the result announced at the beginning of this section.

\begin{Prop}\label{tens6}
The Rokhlin cocycle $x\mapsto S_{\va(x)}|_{\ca_{\lf}}$ is a
coboundary. In other words $\tfs|_{\cb\otimes\ca_{\lf}}$ is
relatively isomorphic to $T\times id_Y|_{\cb\otimes\ca_{\lf}}$.
\end{Prop}
\begin{proof}
Let $\si_\cs|_{\lf}$ be the spectral type of
$\cs|_{L^2(\ca_{\lf})}$ ({\em i.e.\ }$\si_\cs|_{\lf}$ is $\si_\cs$
restricted to $\lf$). When $\chi\in\lf$, the cocycle
$\chi\circ\va$ is a coboundary, that is $\chi\circ\va=f/f\circ T$
$\mu$-a.e.\ for some measurable $f$ of modulus 1. In fact there
exists a measurable selector of transfer functions defined
$\si_\cs|_{\lf}$-a.e.\ (see e.g.\ \cite{Ho-Me-Pa}). This means
that there exists a measurable function $\widetilde F$ of modulus
1 on $X\times \lf$ such that \beq\label{tens5}
\chi(\va(x))=\widetilde{F}(x,\chi)/\widetilde{F}(Tx,\chi)
\;\;\mbox{for}\;\mu\otimes\si_\cs|_{\lf}\mbox{-a.a.}\; (x,\chi)
\eeq (so, for every $g\in L^2 (\ca_{\lf})$, the function $F\in
L^2\xbm\otimes G(g)$ corresponding to $\widetilde{F}$ is
$\tfs$-invariant).

Given $x\in X$, for every $g\in L^2 (\ca_{\lf})$ the action of
$S_{\va(x)}$ on $G(g)$ corresponds through the spectral
isomorphism to the multiplication by the function
$\chi\mapsto\chi(\va(x))$. On the other hand, by the canonical
action of $L^{\infty}(\widehat{G},
\cb(\widehat{G}),\si_{\cs|_{\lf}})$ on $L^2(\ca_{\lf})$, there
also exists a unitary operator $W_x$ on $L^2 (\ca_{\lf})$ whose
restriction to each $G(g)$, for $g\in L^2 (\ca_{\lf})$,
corresponds to the multiplication by the unit-modulus function
$\widetilde{F}(x,\cdot)$. Then the equality~(\ref{tens5}) yields
$$
S_{\va(x)}|_{L^2(\ca_{\lf})}=W_x
W_{Tx}^{-1}\;\;\mbox{for}\;\mu\mbox{-a.a.}\; x.
$$
So, the cocycle $x\mapsto S_{\va(x)}|_{L^2(\ca_{\lf})}$ is a
$T$-coboundary as a cocycle taking values in ${\cal U}(L^2
(\ca_{\lf}))$.

Finally, ${\rm Aut}(\ca_{\lf})$ is naturally identified to a
closed subgroup of the Polish group ${\cal U}(L^2 (\ca_{\lf}))$.
Since $x\mapsto S_{\va(x)}|_{L^2(\ca_{\lf})}$ takes its values in
${\rm Aut}(\ca_{\lf})$, it is still a coboundary as an ${\rm
Aut}(\ca_{\lf})$-valued cocycle.
\end{proof}

In view of Proposition~\ref{p1} we obtain the following result.

\begin{Cor}\label{cor111} If $T$ is ergodic and $\tfs$ is not ergodic then
there is a non-trivial factor $\ca$ of $\cs$ such that
$\tfs|_{\cb\otimes\ca}$ is relatively isomorphic to $T\times
id_Y|_{\cb\otimes\ca}$.\bez\end{Cor}

\subsection{Lifting mild mixing property}
In this section we will show that the triviality of the Rokhlin
cocycle described in Proposition~\ref{tens6} also takes place when
dealing with the mild mixing property, and we give necessary and
sufficient conditions in order that the mild mixing property lift
from $T$ to $\tfs$. Recall that $T$ is mildly mixing if $T$ has no
non-trivial rigid factors and that a factor $\ca$ of $T$ is rigid
if and only if the spectral type of $T|_\ca$ is a Dirichlet
measure.

We will need the following.

\begin{Lemma}\label{lmm1} Assume that  $T$ is mildly mixing.
If $\xi:X\to\T$ is a cocycle and $T_\xi$ has a non-trivial rigid
factor $\ca\subset \cb\ot\cb(\T)$ then there exist a factor ${\cal
A}'$ of $T_\xi$ containing $\ca$ and a mixing sequence $(q_n)$ for
$T$ such that $(T_\xi)^{q_n}\to E(\cdot|\ca')$.
\end{Lemma}
\begin{proof}
From  Proposition~\ref{idemp1} there exist a factor $\ca'$ of
$T_\xi$ containing $\ca$ and a rigid sequence $(q_n)$ for
$T_\xi|_{\ca'}$ such that \beq\label{lmm2} (T_\xi)^{q_n}\to
E(\cdot|\ca'). \eeq It remains to show that $(q_n)$ is a mixing
sequence for $T$. But, since we assume that $T$ is mildly mixing,
no spectral measure of a function in $L^2_0\xbm$ is a Dirichlet
measure, while the maximal spectral type of $T_\xi$ on
$L^2({\ca'})$ is a Dirichlet measure. Thus
$L^2_0({\cb}\otimes\{\emptyset,\T\})\perp L^2({\ca'})$ and the
result follows from~(\ref{lmm2}).
\end{proof}

\begin{Prop}\label{lmm3}
Assume that $T$ is mildly mixing. If $\sigma_{\cs}(\svf)=0$, then
$\tfs$ is also mildly mixing.
\end{Prop}
\begin{proof}
First, we claim that, if some positive measure
$\sigma_1\ll\sigma_{\vhf}$ is a Dirichlet measure, then
$\chi\in\svf$.

Indeed, then there exists $f\in L^2\xbm$ such that
$\si_{f,\vhf}=\si_1$. If we consider $F=f\otimes z\in L^2(X\times
\T,\mu\ot\la)$ (here $z$ denotes the identity function from $\T$
to $\C$), we have $F\circ T_{\chi\circ\varphi}^n=\vhf^nf\ot z$ for
all $n\in \Z$ and it follows that
$\sigma_{F,T_{\chi\circ\varphi}}=\sigma_1$. Since $\si_1$ is a
Dirichlet measure, $F$ is measurable with respect to some rigid
factor $\ca$ of $T_{\chi\circ\varphi}$.

In view of Lemma~\ref{lmm1}, there exist a factor $\ca'$ of
$T_{\chi\circ\varphi}$ containing $\ca$ and a mixing sequence
$(q_n)$ for $T$ such that $T_{\chi\circ\varphi}^{q_n}\to
E(\cdot|\ca')$. In particular $F\circ
T_{\chi\circ\varphi}^{q_n}\to F$ and $\vhf^{q_n}f\to f$, so
$(q_n)$ is not a mixing sequence for $\vhf$. But
Corollary~\ref{lm2} implies then that the cocycle
$\chi\circ\varphi$ is a quasi-coboundary, in other words
$\chi\in\svf$, which proves our claim.

It remains to show that there is no positive measure
$\sigma\ll\stfs$ which is a Dirichlet measure. Suppose the
contrary: then, for some sequence $n_j\to\infty$, $z^{n_j}\to1$
$\sigma$-a.e. It follows that if we set
$A=\{z\in\T:\:z^{n_j}\to1\}$, then $\stfs(A)>0$. Since $T$ is
mildly mixing and thus $\sigma_T(A)=0$, it follows by
Lemma~\ref{l1} that $\sigma_{\cs}(\{\chi\in\widehat{G}:\:
\sigma_{V_{\chi\circ\varphi}}(A)>0\})>0$.

But clearly if $\sigma_{V_{\chi\circ\varphi}}(A)>0$, then the
positive measure $\sigma_{V_{\chi\circ\varphi}}|_A$ is Dirichlet,
hence $\chi$ must belong to $\svf$ by the first part of the proof.
Since $\sigma_{\cs}(\svf)=0$, we obtain a  contradiction.
\end{proof}

\begin{Remark}\label{remlmm3}\em
Supposing only that $T$ is mildly mixing, we get that each rigid
function of $\tfs$ belongs to $L^2\xbm\otimes H_{\svf}$. Indeed,
if $\si_{g,\cs}(\svf)=0$, in view of Lemma~\ref{zima3}, the same
proof gives that there is no Dirichlet measure
$\si\ll\si_{\tfs|_{L^2\xbm\otimes G(g)}}$.

Let us denote by $\ca_{\svf}$ is the factor of $\cs$ corresponding
to the saturated group $\svf$ according to Corollary~\ref{c1}, so
that $H_{\svf}=L^2(\ca_{\svf})$. In other words,
\beq\label{rigidfac} \mbox{\em each rigid factor of $\tfs$ is
contained in $\cb\otimes\ca_{\svf}$.} \eeq
\end{Remark}
\begin{Prop}\label{compl1}
Assume that $T$ is weakly mixing (and $\si_{\cal S}(\svf)>0$).
Then there exists an automorphism $U$ of
$(Y|_{\ca_{\svf}},\ca_{\svf},\nu|_{\ca_{\svf}})$ such that
$\tfs|_{\cb\otimes\ca_{\svf}}$ is isomorphic to $T\times U$.
\end{Prop}
\begin{proof}
By definition, if $\chi\in\svf$ the cocycle $\chi\circ\varphi$ is
cohomologous to a constant $e^{2\pi i t}$. However $T$ is weakly
mixing, so this constant is unique,
hence we can write $t=t(\chi)$ ($t(\chi)\in[0,1)$). Then, as in
the case of $\lf$, there exists a measurable selector of transfer
functions $\chi\mapsto \til{F}(\cdot,\chi)$ defined on
$\si_\cs|_{\svf}$ a.e., equivalently a measurable function
$\til{F}$ of modulus 1 on $X\times \svf$ such that \beq\label{27b}
\chi(\varphi(x))=e^{2\pi it(\chi)}\til{F}(x,\chi)/\til{F}(Tx,\chi)
\;\;\mbox{for}\;\mu\otimes\si_\cs|_{\svf}\mbox{-a.a.}\; (x,\chi).
\eeq In particular, $e^{2\pi it(\chi)}=u(\chi)$,
$\si_\cs|_{\svf}$-a.e., where $u$ is a measurable function of
modulus 1 on $\svf$.

Then, as in the proof of Proposition~\ref{tens6}, we deduce that
there exist unitary operators $U$ and $W_x$ ($x\in X$) of
$L^2(\ca_{\svf})$ corresponding to the multiplication by $u$ and
$\til{F}(x,\cdot)$ respectively, so that \beq\label{27}
S_{\va(x)}|_{\svf}=UW_xW_{Tx}^{-1}\quad\mu\mbox{-a.e.}~x. \eeq We
have to show that $U$ corresponds to an automorphism. Let us
consider the $(T\times T, \mu\ot\mu)$-cocycle $(x_1,x_2)\mapsto
S_{\va(x_2)}S_{\va(x_1)}^{-1}|_{\svf}$. In view of~(\ref{27}) (and
the fact that the operators under consideration commute), it is a
coboundary with the transfer operator map $(x_1,x_2)\mapsto
W_{x_2}W_{x_1}^{-1}$, whence, as in the proof of
Proposition~\ref{tens6}, it is also a coboundary as a cocycle with
values in  Aut$(\ca_{\svf})$. Thus there exists a measurable map
$(x_1,x_2)\mapsto V_{x_1,x_2}\in\hbox{Aut}(\ca_{\svf})$ with
$$
S_{\va(x_2)}S_{\va(x_1)}^{-1}|_{\svf}=V_{x_1,x_2}V_{Tx_1,Tx_2}^{-1}\quad
\mbox{for }\mu\otimes\mu\mbox{-a.a.}~(x_1,x_2).
$$
Since $T$ is weakly mixing, $T\times T$ is ergodic and therefore
the two transfer operator maps must coincide up to a constant.
More precisely, $W_{x_1}W_{x_2}^{-1}V_{x_1,x_2}$ is $T\times
T$-invariant, so there exists a unitary operator $V$ of
$L^2(\ca_{\svf})$ such that
$$W_{x_1}W_{x_2}^{-1}V_{x_1,x_2}=V\quad\mbox{for }
\mu\otimes\mu\mbox{-a.a.}~(x_1,x_2).$$ By selecting $x_1$ so that
the above equality is true for $\mu$-a.e.\ $x_2$, we obtain
$$V_{x_1,x}=W_{x}W_{x_1}^{-1}V\quad\mbox{for }\mu\mbox{-a.a.}~x.$$
Then the map $x\mapsto V_{x_1,x}\in\hbox{Aut}(\ca_{\svf})$ is also
a transfer operator map for the equation~(\ref{27}):
$$ UV_{x_1,x}V_{x_1,Tx}^{-1}=UW_xW_{Tx}^{-1}=S_{\va(x)}|_{\svf}
\quad\mu\mbox{-a.e.}~x.$$ Therefore $U\in\hbox{Aut}(\ca_{\svf})$,
$x\mapsto S_{\va(x)}$ is cohomologous to the constant $U$ in
Aut$(\ca_{\svf})$, and the result follows.
\end{proof}

\begin{Cor}\label{compl2} Assume that $T$ is mildly mixing. Then
$\tfs$ is not mildly mixing if and only if there exists a
non-trivial factor $\ca$ of $\cal S$ and an automorphism $U$ of
$(Y|_{\ca},\ca,\nu|_{\ca})$ which is not mildly mixing such that
$\tfs|_{\cb\otimes\ca}$ is isomorphic to $T\otimes U$.
\end{Cor}
\begin{proof}
In view of (\ref{rigidfac}), if $\tfs$ is not mildly mixing,
neither is $\tfs|_{\cb\otimes\ca_{\svf}}$ and we apply
Proposition~\ref{compl1}. Then $T\times U$ is not mildly mixing
and, as $T$ is mildly mixing, $U$ cannot be mildly mixing. The
other direction  is clear.\end{proof}

\begin{Remark} \em It turns out that in
Corollary~\ref{compl2} we can replace $U$ non mildly mixing by $U$
rigid. Indeed,  in the proof of Proposition~\ref{compl1}, $U$
corresponds to the  multiplication by $u(\chi)$, and given a rigid
factor $\ca$ of $U$ defined by $U^{n_j}h\to h$ for some sequence
$(n_j)$, we have that $L^2(\ca)$ is the spectral subspace of ${\cal
S}|_{\ca_{\svf}}$ corresponding to $\{\chi\in\svf:\:
u(\chi)^{n'_j}\to1\}$ for some subsquence; in particular, $\ca$ is
also $\cal S$-invariant. Moreover, the cocycle
$S_{\va(x_2)}S_{\va(x_1)}^{-1}|_{\svf}$is still cohomologous to
the constant $U$ in the closed subgroup of ${\rm Aut}(\ca_{\svf})$
of all automorphisms corresponding to multiplications by
unit-modulus functions (in the spectral representation of
$\cs|_{\ca_{\svf}}$). So, the automorphisms $V_{x_1,x_2}$ can be
taken in this subgroup and hence preserving the invariant
subspaces of $\cal S$.  Then $\cb\ot \ca$ is preserved by the
conjugation automorphism and we have relative isomorphism of
$\tfs|_{\cb\otimes\ca}$ with $T\otimes U|_{\cb\otimes\ca}$.
\end{Remark}

We now show that under the recurrence property of $\varphi$ the
converse of Proposition~\ref{lmm3} holds.

Let $\si\in M^+(\widehat{G})$. Denote by ${\cal U}(\si)$ the group
of measurable functions of modulus~1 defined on $\widehat{G}$,
modulo equality $\si$-a.e. We endow ${\cal U}(\si)$ with the
$L^2(\si)$-topology, which makes it a Polish group. Given $g\in
G$, we still denote by $\tilde{g}$ the function $\chi\mapsto
\chi(g)$ taken as an element of ${\cal U}(\si)$. Then we define
$\va_\si$ from $X$ to ${\cal U}(\si)$ by setting
$$
\va_\si(x)(\chi)=\chi(\varphi(x))\;\;\mbox{for
each}\;\chi\in\widehat{G},
$$
{\em i.e.\ }$\va_\si$ is the composition of $\va$ and of the map
$g\mapsto \tilde{g}$. As the latter map is a continuous group
homomorphism, it is clear from the definition (\ref{rec1}) that
\beq\label{rec10} \mbox{\em if $\va$ is recurrent then so is
$\va_\si$.} \eeq
\begin{Prop}\label{compl3} If $T$ is mildly mixing, then
$\tfs$ is mildly mixing if and only if $\sigma(\svf)=0$ for each
positive measure $\sigma\ll\sigma_{\cal S}$ such that the cocycle
$\va_\si$ is recurrent.

In particular, if $\va$ is recurrent and  $\tfs$ is mildly mixing
then $\sigma_{\cal S}(\svf)=0$.\end{Prop}
\begin{proof}
We keep the notation as in the proof of Proposition~\ref{compl1}:
$U$ is the automorphism of
$(Y|_{\ca_{\svf}},{\ca_{\svf}},\nu|_{\ca_{\svf}})$ corresponding
to the unit-modulus function $u$ on $\svf$, and the equation
(\ref{27b}) may now be written as
$$
\va_\si(\chi)=u(\chi)\til{F}(x,\chi)/\til{F}(Tx,\chi)
\;\;\mbox{for}\;\mu\otimes\si_\cs|_{\svf}\mbox{-a.a.}\; (x,\chi).
$$

Suppose that $\va_\si$ is recurrent for some positive measure $\sigma\ll\sigma_{\cal S}$ with $\sigma(\svf)>0$.
We can then assume that $0<\sigma\ll\sigma_{\cal S}|_{\svf}$. Then
the constant cocycle $u$ restricted to ${\cal U}(\si)$ is
cohomologous to $\va_\si$ and it is also recurrent by
(\ref{rec2}). Thus, in view of (\ref{rec3}), there is a sequence
$(n_j)$ with $u^{n_j}\to 1$ in ${\cal U}(\sigma)$, whence
$U^{n_j}h\to h$ for each function $h$ such that $\sigma_{h,{\cal
S}}\ll\sigma$. It follows that $U$ is not mildly mixing.

For the other direction, if $\tfs$ is not mildly mixing, then $U$
is not mildly mixing and we find conversely that $u$ restricted to
${\cal U}(\sigma_{h,{\cal S}})$ is recurrent, for some non-zero
$h\in L^2_0(\ca_{\svf})$. Then $\va_{\si_{h,{\cal S}}}$ is also recurrent and $\si_{h,{\cal S}}(\svf)>0$.

The second assertion follows then from (\ref{rec10}): if $\va$ is
recurrent, then the cocycle $\va_{\si_\cs}$ is also recurrent.
\end{proof}

\subsection{Lifting mixing and multiple mixing}
We give here two corollaries of results from \cite{Le-Pa}.

\begin{Prop}[\cite{Le-Pa}]\label{mix1}
Assume that $T$ is mixing. If $\sigma_{\cs}(\svf)=0$ then $\tfs$
is mixing. Conversely, if $\sigma_{\cs}(\svf)>0$ and $\va$ is
recurrent then $\tfs$ is not mixing. \bez\end{Prop}

\begin{Cor}\label{compl4} Assume that $T$ is mixing. Then
$\tfs$ is not mixing if and only if there exists a non-trivial
factor $\ca$ of $\cal S$ and an automorphism $U$ of
$(Y|_{\ca},\ca,\nu|_{\ca})$ which is not mixing such that
$\tfs|_{\cb\otimes\ca}$ is isomorphic to $T\otimes U$.
\end{Cor}
\begin{proof}The proof of Proposition~\ref{mix1} in \cite{Le-Pa} (Theorem 7.1) shows
actually that if $T$ is mixing and $\sigma_{h,{\cal S}}(\svf)=0$,
then $\hat\sigma_{f\otimes h,\tfs}(n)\to 0$ for each $f\in
L^2\xbm$. Therefore if $\tfs$ is not mixing, we must have a
function $F\in L^2_0(X,\cb,\mu)\ot H_{\svf}$ whose spectral
measure does not vanish at infinity, whence
$\tfs|_{\cb\ot{\ca_{\svf}}}$ is not mixing. Then we can apply
Proposition~\ref{compl1} and the result follows exactly as for
Corollary~\ref{compl2}.
\end{proof}

\begin{Prop}[\cite{Le-Pa}]\label{mix2}
Assume that $T$ is $r$-fold mixing and that $\va$ is recurrent. If
$\tfs$ is mildly mixing then it is also $r$-fold mixing.
\bez\end{Prop}

Now, the corollary below directly follows from
Proposition~\ref{mix2} and Proposition~\ref{compl3}.

\begin{Cor}\label{mix3}
Assume that $T$ is $r$-fold mixing and $\va$ is recurrent. Then
$\tfs$ is $r$-fold mixing if and only if $\sigma_{\cs}(\svf)=0$.
\bez\end{Cor}

\vspace{5mm}

\noindent Mariusz Lema\'nczyk\\
Faculty of Mathematics and Computer Science\\
Nicolaus Copernicus University\\
12/18 Chopin street, 87--100 Toru\'n,
Poland\\
mlem@mat.uni.torun.pl

\vspace{2mm}

\noindent Fran\c{c}ois Parreau\\
Laboratoire d'Analyse, G\'eom\'etrie
et Applications, UMR 7539\\
Universit\'e Paris 13 et CNRS\\
99, av. J.-B.\ Cl\'ement, 93430 Villetaneuse, France\\
parreau@math.univ-paris13.fr

\end{document}